\documentclass[11pt,twoside]{atmp}

\usepackage{amsmath,amssymb}
\usepackage{graphicx}
\usepackage{subfigure}
\usepackage[all]{xy}
\usepackage{color}

\newcommand{\Z}{\mathbb Z}
\newcommand{\F}{\mathbb F}

\DeclareMathOperator{\Ext}{Ext}
\newcommand{\A}{\mathcal A}
\newcommand{\Oo}{\widetilde{\Omega}}

\newtheorem{thm}{Theorem}[section]
\newtheorem{cor}[thm]{Corollary}

\theoremstyle{remark}
\newtheorem*{remark}{Remark}

\begin{document}

\title[MString of $BE_8$ and $BE_8\times BE_8$]{The String Bordism of $BE_8$ and $BE_8\times BE_8$ through dimension $14$}


\author{Michael A.~Hill}

\address{Department of Mathematics \\
University of Virginia \\
Charlottesville, VA 22904}  
\addressemail{mikehill@virginia.edu}

\begin{abstract}
We compute the low dimensional String bordism groups $\Oo_{k}^{String}BE_8$ and $\Oo_{k}^{String}(BE_8\times BE_8)$ using a combination of Adams spectral sequences together with comparisons to the Spin bordism cases.
\end{abstract}

\maketitle
\section{Introduction}

In this paper, we will answer a question posed by Hisham Sati about the low dimensional String bordism groups of spaces of particular interest to string theorists. We specifically compute the bordism groups of $K(\Z,4)$; of $BE_8$, the classifying space for principal bundles for the exceptional Lie group $E_8$; and of their cartesian squares. These computations have applications to various models string theory built out of $E_8$ bundles, many of which will be spelled out more fully in a future paper.

This computation generalizes computations of Stong of the Spin bordism groups of $K(\Z,4)$ and of Edwards of the Spin bordism of $BE_8\times BE_8$ \cite{StongMSpinKZ4, EdwardsMSpinBE8BE8}. In particular, we see that the String bordism of $K(\Z,4)$ injects into the Spin bordism of $K(\Z,4)$ through dimension at least $14$. This means that many of the cohomological invariants used to detect bordism classes of Spin manifolds apply equally well here. In particular, the comparison of Adams spectral sequences used to compute the $10^{\text{th}}$ String bordism group of $BE_8$ shows that the Landweber-Stong invariant detects this String bordism class \cite{LandweberStongBilinearForm, DiaconescuMooreWitten}.

Our main results are summarized in the following two theorems.

\begin{thm}\label{thm:MStringBE8}
Through dimension $15$, we have the following reduced String bordism groups.

\begin{table}[h]
\centering
\begin{tabular}{|c|c|c|c|c|c|c|c|c|c|c|c|}
\hline
k & 4 & 5 & 6 & 7 & 8 & 9 & 10 & 11 & 12 & 13 & 14 \\ \hline
$\Oo_k^{String}BE_8$ & $\Z$ & 0 & 0 & 0 & $\Z$ & $\Z/2$ & $\Z/2$ & $0$ & $\Z^2$ & 0 & 0 \\ \hline
\end{tabular}
\end{table}
\end{thm}

The String bordism groups of $BE_8\times BE_8$ contain two copies of the String bordism groups of $BE_8$, together with the String bordism of the space $BE_8\wedge BE_8$. These new groups are given in the following theorem.

\begin{thm}\label{thm:MStringBE8BE8}
Through dimension $15$, we have the following reduced String bordism groups.

\begin{table}[h]
\centering
\begin{tabular}{|c|c|c|c|c|c|c|c|c|c|c|c|}
\hline
k & 4 & 5 & 6 & 7 & 8 & 9 & 10 & 11 & 12 & 13 & 14 \\ \hline
$\Oo_k^{String}(BE_8\wedge BE_8)$ & 0 & 0 & 0 & 0 & $\Z$ & 0 & $\Z/2$ & 0 & $\Z^2\oplus\Z/3$ & 0 & $(\Z/2)^2$ \\ \hline
\end{tabular}
\end{table}
\end{thm}

\section{Computational Reductions}
\subsection{Reduction to $tmf$ and $K(\Z,4)$}
Our actual computation will be of the $tmf$-homology of  $K(\Z,4)$. We proceed in this section to explain why this computation is sufficient.

The $\sigma$-orientation of Ando-Hopkins-Strickland is the primary tool. This is an $E_\infty$ ring map $MString\to tmf$ that refines the Witten genus \cite{AndoHopkinsStricklandHinfty}.

\begin{thm}
The $\sigma$-orientation is $15$-connected.
\end{thm}

\begin{proof}
We argue this locally, showing that it is true at $2$, at $3$, and with $6$ inverted.

For $p=2$, we show this is true in cohomology. Since the $\sigma$ orientation is a ring map, composition with the unit map shows that
\[
\sigma^\ast\colon\Z=H^0(tmf)\to H^0(MString)=\Z
\]
is an isomorphism.

Computations of Stong and Bahri-Mahowald show that as a module over $\A$, the Steenrod algebra,
\[
H^\ast(MString)=\A/\!/\A(2)\oplus M,
\]
where $\A(2)$ is the subalgebra of the Steenrod algebra generated by $Sq^1$, $Sq^2$, and $Sq^4$, and where $M$ is a $19$-connected $\A$-module \cite{StongBOn, BahriMahowaldMString}. Hopkins and Mahowald have shown that $H^\ast(tmf)=\A/\!/\A(2)$ \cite{HopkinsMahowaldEO2}, so since $\sigma^\ast$ is an isomorphism on $H^0$, we conclude that
\[
\A/\!/\A(2)=H^\ast(tmf)\hookrightarrow H^\ast(MString).
\]
In particular, the smallest degree element of the cokernel is in dimension $20$, giving the result for $p=2$.

The remaining cases are simpler. For $p=3$, computations of Hovey and Ravenel show that the $\sigma$-orientation is $15$-connected \cite{HoveyRavenelBO8}. When $6$ is a unit, there is no torsion in the homotopy of $MString$, and in degrees less than $17$, there are polynomial generators in degree $8$, $12$, and $16$. The $\sigma$-orientation sends the generator in degree $8$ to $c_4$ and the one in degree $12$ to $c_6$. This gives the final case.
\end{proof}

Bott and Samelson showed that the Postnikov section $E_8\to K(\Z,3)$ is $14$-connected \cite{BottSamelson}. This implies that the Postnikov section $BE_8\to K(\Z,4)$ is an isomorphism in homotopy in degrees less than $16$. In particular, we see that if $X^{[k]}$ denotes the $k$-skeleton of $X$, then
\[
K(\Z,4)^{[15]}\simeq BE_8^{[15]}.
\]
Since the spectra $tmf$ and $MString$ are $(-1)$-connected, the inclusion of the $15$-skeleton induces an isomorphism in homology through degree $14$. Combined with the previous theorem, we conclude the following.

\begin{cor}
As graded rings,
\[
\widetilde{tmf}_k K(\Z,4) \cong \Oo_k^{String} BE_8
\]
for $k < 15$, where the isomorphism is the composite of the Postnikov section with the $\sigma$-orientation.
\end{cor}

This is a substantial simplification, since $tmf$-homology at all primes is computable using the Adams spectral sequence.

\subsection{Forms of the Adams Spectral Sequence}

Hopkins and Mahowald have shown that $H^\ast(tmf)=\A/\!/\A(2)$, where $\A(2)$ is the subalgebra of the Steenrod algebra generated by $Sq^1$, $Sq^2$, and $Sq^4$. Standard change-of-rings arguments then show the following.

\begin{thm}[\cite{HopkinsMahowaldEO2}]\label{thm:tmfASSat2}
If $X$ is a finite spectrum, then there is an Adams spectral sequence of the form
\[
E_2=\Ext_{\A(2)}^{s,t}\big(\tilde{H}^\ast(X;\F_2),\F_2\big)\Longrightarrow \widetilde{tmf}_{t-s}X_{2}^{\wedge}.
\]
\end{thm}

At $p=3$, the cohomology of $tmf$ is not a cyclic $\A$-module, so similar techniques no longer work. By considering instead a slightly refined version of the Adams spectral sequence, described by Baker and Lazarev \cite{BakerLazarev}, we can still build an analogous spectral sequence.

There is a Hopf algebra $\tilde{\A}(1)$ analogous to $\A(2)$ used to compute $tmf$-homology at $p=3$. This is most easily expressed in the dual formulation:
\[
\tilde{\A}(1)_\ast=\F_3[\xi_1]/\xi_1^3\otimes E(\tau_0,\tau_1,a_2),
\]
where $\xi_1$, $\tau_0$, and $\tau_1$ arise from the elements of the dual Steenrod algebra of the same name, and where the coproduct on $a_2$ is given by
\[
\psi(a_2)=a_2\otimes 1+1\otimes a_2 + \xi_1\otimes\tau_1-\xi_1^2\otimes\tau_0.
\]

\begin{thm}[\cite{HilltmfBSigma3}]\label{thm:tmfASSat3}
If $X$ is a finite spectrum, then there is an Adams spectral sequence of the form
\[
E_2=\Ext_{\tilde{\A}(1)}^{s,t}\big(\tilde{H}^\ast(X;\F_3),\F_3\big)\Longrightarrow \widetilde{tmf}_{t-s}X_{3}^{\wedge}.
\]
\end{thm}

At $p>3$, the computations are simplified by the lack of $p$-torsion and a splitting of $tmf$ into spectra built out of $BP$.

\begin{thm}[\cite{HMtmf, Rezk512Notes}]
At primes larger than $3$, we have a splitting
\[
tmf_{(p)}=\bigvee BP\langle 2\rangle.
\]
\end{thm}

Since $H^\ast(BP\langle 2\rangle)=\A/\!/ E(Q_0,Q_1,Q_2)$, where $Q_i$ is the $i^{\text{th}}$ Milnor primitive, standard change-of-rings arguments allow us to compute each summand of $tmf_\ast X$ arising from this splitting.

\begin{thm}[\cite{GreenBook}]
If $X$ is a finite spectrum then there is an Adams spectral sequence of the form
\[
E_2=\Ext_{E(Q_0,Q_1,Q_2)}^{s,t}\big(\tilde{H}^\ast(X;\F_p),\F_p\big)\Longrightarrow \widetilde{BP\langle 2\rangle}_{t-s} X_{p}^{\wedge}
\]
\end{thm}

At the prime $2$, we will also have need of the $ko$-homology of $K(\Z,4)$.
Since $H^\ast(ko)=\A/\!/\A(1)$, where $\A(1)$ is the subalgebra of the Steenrod algebra generated by $Sq^1$ and $Sq^2$, a change-of-rings argument similar to that for $tmf$ allows us to compute $ko$-homology.

\begin{thm}
If $X$ is a finite spectrum, then there is a spectral sequence of the form
\[
E_2=\Ext_{\A(1)}^{s,t}\big(\tilde{H}^\ast(X;\F_2),\F_2\big)\Longrightarrow \widetilde{ko}_{t-s} X_{2}^\wedge.
\]
\end{thm}

In cohomology, the effect of the canonical map $tmf\to ko$ is the quotient
\[
H^\ast(ko)=\A/\!/\A(1)\to \A/\!/\A(2)=H^\ast(tmf).
\]
This means that the induced map on Adams spectral sequences is given on $E_2$ terms by the canonical map
\[
\Ext_{\A(2)}\big(\tilde{H}^\ast(X;\F_2),\F_2\big)\to\Ext_{\A(1)}\big(\tilde{H}^\ast(X;\F_2),\F_2\big)
\]
induced by the inclusion of $\A(1)$ into $\A(2)$.

Many of our Adams spectral sequence arguments will be made more clear with pictures of the corresponding $E_2$ pages. In all of these charts, the horizontal axis represents $t-s$, the difference between the internal grading and the $\Ext$ degree, and this corresponds to the stem of the target. This means that all of the groups in a column reassemble to give a single homotopy group. The vertical direction represents $s$, the $\Ext$ degree, and a $d_r$ differential decreases $t-s$ by $1$ and increases $s$ by $r$. Additionally, it the pictures that follow, a dot represents a copy of $\Z/p$, and vertical lines connecting dots represents multiplication by a class $v_0$ from the Adams spectral sequence of the sphere. This class detects multiplication by $p$, reflecting a non-trivial additive extension between the linked copies of $\Z/p$. In particular, a class which is $v_0$-torsion free represents a copy of $\Z_p$.

Since $MString$ and $BE_8$ are finite type, we know that copies of $\Z_p$ in the $p$-completion arise from copies of $\Z_{(p)}$ in the localization. We will therefore blur the distinction in what follows, stating without further comment that $v_0$-towers yield copies of $\Z_{(p)}$.

\subsection{Cohomology of $K(\Z,4)$}
The final ingredient is the cohomology of Eilenberg-MacLane spaces, worked out by Cartan and Serre \cite{CartanKpin, SerreKpin}.

\begin{thm}
The cohomology of $K(\Z,n)$ with coefficients in $\F_p$ is the free graded commutative unstable algebra over the Steenrod algebra generated by a class $\iota_n$ in degree $n$ subject to the relation that $\beta\iota_n=0$.
\end{thm}

In particular, $H^\ast(K(\Z,4);\F_p)$ is generated by classes $i_4$ and $\mathcal P^I \iota_4$ for all admissible $I$ of excess less than $4$ that do not end in $\beta$. This is a very harsh restriction.

\begin{cor}
For $p>5$, the only cohomology classes of dimension less than $16$ are $\iota_4$, $\iota_4^2$, and $\iota_4^3$.

For $p=5$, there are additionally the classes $\mathcal P^1 \iota_4$ and $\beta\mathcal P^1\iota_4$ in degrees $12$ and $13$ respectively.

For $p=3$, there are classes $\iota_4$, $\mathcal P^1\iota_4$, $\beta\mathcal P^1\iota_4$, $\iota_4^2$, $\iota_4^3=\mathcal P^2\iota_4$, $\iota_4\mathcal P^1\iota_4$, and $\iota_4\beta\mathcal P^1\iota_4$.

For $p=2$, there are classes $\iota_4$, $Sq^2\iota_4$, $Sq^3\iota_4$, $\iota_4^2=Sq^4\iota_4$, $Sq^4Sq^2\iota_4$, $\iota_4Sq^2\iota_4$, $Sq^5Sq^2\iota_4$, $\iota_4Sq^3\iota_4$, $(Sq^2\iota_4)^2=Sq^6Sq^2\iota_4$, $\iota_4^3$, $Sq^6Sq^3\iota_4$, $Sq^2\iota_4Sq^3\iota_4$, $(Sq^3\iota_4)^2=Sq^7Sq^3\iota_4$, $\iota_4Sq^4Sq^2\iota_4$, $\iota_4^2Sq^2\iota_4$ $\iota_4^2Sq^3\iota_4$, and $\iota_4Sq^5Sq^2\iota_4$.
\end{cor}

We remark that for $p>5$, the Bockstein spectral sequence collapses, showing that the integral cohomology of $K(\Z,4)$ is $p$-torsion free through dimension $15$.

The modules for $p=3$ and $p=2$ are best understood via pictures. In these, the vertical direction indicates the dimension, and dots represent basis elements. In the picture for $p=3$, curved lines represent the action of $\mathcal P^1$ and straight lines represent the action of $\beta$. The module is presented in Figure~\ref{fig:Moduleat3}. For $p=2$, Theorem~\ref{thm:tmfASSat2} shows us that we need only understand the cohomology as a module over $\A(2)$. For this, straight lines indicate $Sq^1$, curved lines indicate $Sq^2$ and brackets indicate $Sq^4$. The module (with a slightly nicer basis) is depicted in Figure~\ref{fig:Moduleat2}.

\begin{figure}[ht]
\centering
\subfigure[$p=3$]{
\label{fig:Moduleat3}
\includegraphics{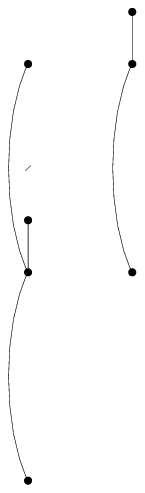}}
\hspace{3cm}
\subfigure[$p=2$]{
\label{fig:Moduleat2}
\includegraphics{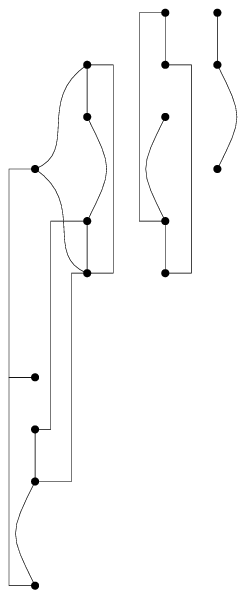}}
\caption{$H^k(K(\Z,4);\F_p)$ for $k<16$ as an $\A$-module}
\end{figure}

This collection of data is sufficient to carry out almost all of the required computations.

\section{The Groups $\Oo_k^{String}(BE_8)$ for $k<15$}

\subsection{Computation with $6$ Inverted}

When $6$ is inverted, the spectrum $tmf$ has no torsion. This makes running the Atiyah-Hirzebruch spectral sequence for $\widetilde{tmf}_k K(\Z,4)$ much simpler.

\begin{thm}\label{thm:BE8odd}
For $k<15$, we have
\[
\Oo_k^{String}BE_8\otimes\Z[\tfrac{1}{6}]=\begin{cases}
\Z[\tfrac{1}{6}] & \text{when }k=4, 8, \\
\Z[\tfrac{1}{6}]\times\Z[\tfrac{1}{6}] & \text{when }k=12, \\
0 & \text{otherwise.}
\end{cases}
\]
\end{thm}

\begin{proof}
Localized at primes larger than $5$, there is no torsion in the integral homology of $K(\Z,4)$ in our range. For degree reasons, the Atiyah-Hirzebruch spectral sequence then collapses, giving the result.

At $p=5$, the Atiyah-Hirzebruch spectral sequence becomes less transparent:
\[
H_{12}(K(\Z,4);\Z[\tfrac{1}{6}])=\Z/5\oplus \Z[\tfrac{1}{6}],
\]
and this propagates to give a possible $5$-torsion class in $\Oo_{12}^{String} K(\Z,4)$. Using instead the Adams spectral sequence, we see that in fact $\Oo_{12}^{String} K(\Z,4)[\tfrac{1}{6}]$ is torsion free, since the class $\beta \mathcal P^1$ is $Q^1$ on integral classes and detects $v_1$. At $p=5$, the class $v_1$ is the class $c_4$ in $tmf$ \cite{Rezk512Notes}, showing that the $12^{\text{th}}$ String bordism group is in fact torsion free.
\end{proof}

\begin{remark}
The possible $5$-torsion in $\Oo_{12}^{String} K(\Z,4)$ is exactly analogous to the possible $3$-torsion Stong encountered computing $\Oo_{8}^{Spin} K(\Z,4)$. Stong elegantly avoided the use of the Adams spectral sequence by appealing to characteristic classes. A similar argument applies here, relating the mod $5$ reduction of $p_2$ to the first Wu class.
\end{remark}

\subsection{Computation at $p=3$}

The space $K(\Z,4)$ splits stably at $p=3$, due to the presence of the automorphism given by multiplication by $-1$. This acts on $i_4$ by $-1$, and the stable splitting is reflected in the cohomology by splitting the cohomology into $\pm 1$ eigenspaces. In particular, the two summands of Figure~\ref{fig:Moduleat3} arise from different stable summands, and therefore they do not interact in the Adams spectral sequence.

The relevant $\Ext$ groups for the two modules making up $H^\ast(K(\Z,4);\F_3)$ are easy to compute using the long exact sequences in $\Ext$ induced by the inclusions of the skeleta \cite{HilltmfBSigma3}, and the analogue of the Adams $E_2$ term is given in Figure~\ref{fig:Extat3}.

\begin{figure}[ht]
\centering
\includegraphics{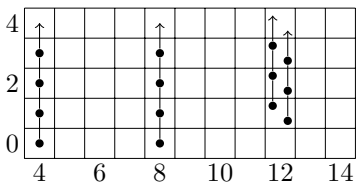}
\caption{An Adams $E_2$ term for $\widetilde{tmf}_k K(\Z,4)$, $k<15$}
\label{fig:Extat3}
\end{figure}

We see that for degree reasons, there are no differentials and no possible extensions. In particular, there are no torsion summands in this range, and we conclude that at $3$, the results are essentially the same as for large primes.

\begin{thm}\label{thm:BE8at3}
For $k<15$, we have that
\[
\Oo_k^{String} BE_8\otimes\Z_{(3)}=\begin{cases}
\Z_{(3)} & \text{when }k=4,8, \\
\Z_{(3)}^2 & \text{when }k=12, \\
0 & \text{otherwise.}
\end{cases}
\]
\end{thm}

\subsection{Computation at $p=2$}

Using a program written by Robert Bruner \cite{BrunerExt90s}, we compute
\[
\Ext_{\A(2)}^{s,t}\big(\tilde{H}^\ast(K(\Z,4);\F_2),\F_2\big).
\]
The result is depicted in Figure~\ref{fig:Extat2}. In this figure, lines of slope $1$ are multiplication by $\eta$, the generator of $\pi_1^s(S^0)$. Since the stable homotopy groups of spheres are the cobordism groups of framed manifolds, we can interpret this class as a framed manifold: $S^1$ together with the non-bounding framing.

\begin{figure}[ht]
\centering
\includegraphics{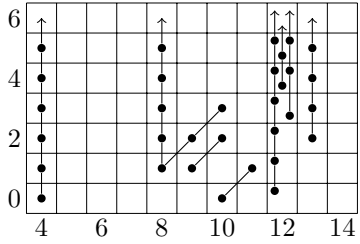}
\caption{The Adams $E_2$ term for $\widetilde{tmf}_k K(\Z,4)$, $k<15$}
\label{fig:Extat2}
\end{figure}

There is the possibility for a $d_2$ differential in degrees $10$ and $11$. Using Stong's computation of $\Oo_{11}^{\text{Spin}} K(\Z,4)$, we can show that there is such a differential.

\begin{thm}[\cite{StongMSpinKZ4}]
We have
\[
\Oo_{10}^{\text{Spin}} K(\Z,4)=\Z/2\oplus\Z/2\text{  and  }\Oo_{11}^{\text{Spin}} K(\Z,4)=0.
\]
\end{thm}

We can show our desired differential by mapping down to $ko$. There is a commutative square
\[
\xymatrix{{MString}\ar[d]_{\sigma}\ar[r] & {MSpin}\ar[d]^{\hat{A}} \\
{tmf} \ar[r] & {ko.}}
\]
The $\hat{A}$-genus is the projection onto one of the summands of $MSpin$ described by Anderson, Brown, and Peterson \cite{AndersonBrownPetersonSpin}. Through dimension $10$, the splitting looks like
\[
MSpin=ko \vee ko[8] \vee ko[10],
\]
where $ko[n]$ denotes the $(n-1)$-connected cover of $ko$.

Since $K(\Z,4)$ begins in dimension $4$, we conclude that the $\hat{A}$-genus induces an isomorphism
\[
\Oo_{k}^{Spin}(K(\Z,4))\xrightarrow{\simeq} \widetilde{ko}_kK(\Z,4)
\]
for $k\leq 11$.

The Adams $E_2$ term for $\widetilde{ko}_\ast K(\Z,4)$ is very easy to compute directly, and is presented in Figure~\ref{fig:Extkoat2} through dimension $14$.

\begin{figure}[ht]
\centering
\includegraphics{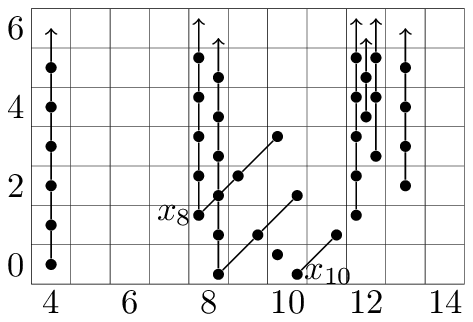}
\caption{The Adams $E_2$ term for $ko_\ast K(\Z,4)$}
\label{fig:Extkoat2}
\end{figure}

It is similarly easy to see that the map on Adams $E_2$ terms induced by $tmf\to ko$ is injective through dimension $14$.

Stong's computation that $\Oo_{11}^{Spin}K(\Z,4)$ is $0$ shows that the class $\eta x_{10}$ cannot survive the Adams spectral sequence. For degree reasons, the only way to remove this class is for it to support a differential, and the only possibility is
\[
d_2(\eta x_{10})=\eta^2 x_8.
\]
Since the Adams spectral sequence for $ko$-homology is a spectral sequence of modules over the Adams spectral sequence for $ko_\ast$, we conclude that we also have a differential of the form
\[
d_2(x_{10})=\eta x_8.
\]

Naturality of the Adams spectral sequence implies that the same must be true for $tmf$.

\begin{cor}
There is a $d_2$ differential in dimensions $10$ and $11$ in the Adams spectral sequence for $\widetilde{tmf}_\ast K(\Z,4)$.
\end{cor}

In the homotopy of $ko$, there is a Massey product relation linking the generator of the $4$ stem with the lower generators:
\[
ko_4=\Z\cdot\langle 2,\eta,\eta^2\rangle.
\]
This relationship propagates, and in particular, we know that the generator of $13$ stem in the Adams $E_2$ term for $\widetilde{ko}_\ast K(\Z,4)$ is linked by a bracket of the form $\langle 2,\eta, -\rangle$ to $\eta x_{10}$. Adams differentials satisfy a kind of Leibnitz rule \cite{GreenBook}, and this allows us to conclude that there is a $d_2$ differential on the $13$-dimensional generator, hitting the generator of the $\Z$-tower beginning in filtration $4$ in dimension $12$ modulo indeterminacy. For this bracket, the indeterminacy is the submodule generated by $v_0$ times each of the generators in this degree, and thus cannot cancel out the piece of the differential hitting the aforementioned generator.

\begin{cor}
In dimension $12$, $\Oo_{k}^{String} K(\Z,4)\otimes\Z_{(2)}$ is torsion free.
\end{cor}

Using an analysis built out of the fact that $H^{13}(K(\Z,4);\Z)=\Z/20$, Francis has shown that the target of the differential hits non-trivial multiples of all of the classes in that degree \cite{JNKFBSpin}.

For degree reasons, there are no further differentials possible and no possible extensions. We conclude the following theorem.

\begin{thm}\label{thm:BE8at2}
For $k<15$, we have
\[
\Oo_k^{String} BE_8\otimes \Z_{(2)}=\begin{cases}
\Z_{(2)} & \text{when }k=4,8, \\
\Z/2 & \text{when }k=9,10, \\
\Z_{(2)}^2 & \text{when }k=12, \\
0 & \text{otherwise.}
\end{cases}
\]
\end{thm}

Our analysis of the Adams spectral sequences for $\widetilde{ko}_\ast K(\Z,4)$ and $\widetilde{tmf}_\ast K(\Z,4)$ shows an additional result, since all of the differentials for $\widetilde{ko}_\ast K(\Z,4)$ lift to differentials for $\widetilde{tmf}_\ast K(\Z,4)$.

\begin{cor}\label{cor:injective}
The map
\[
\Oo_{k}^{String}(BE_8)\otimes\Z_{(2)}\to\Oo_{k}^{Spin}(BE_8)\otimes\Z_{(2)}
\]
is injective for $k<14$.
\end{cor}

Theorems~\ref{thm:BE8odd}, \ref{thm:BE8at3}, and \ref{thm:BE8at2} together give a restatement of Theorem~\ref{thm:MStringBE8}. In each degree, we simply find a finitely generated abelian group which has the correct localizations with respect to each prime, and this yields the table given in the introduction. Since the String bordism groups with $2$ inverted are torsion free, Corollary~\ref{cor:injective} implies that integrally,
\[
\Oo_k^{String}(BE_8)\to\Oo_k^{Spin}(BE_8)
\]
is injective for $k<14$.

In particular, in dimension $10$, a dimension of particular interest to String theorists, the cobordism invariant discussed by Diaconescu-Moore-Witten can also be used to detect cobordism classes of String $10$-manifolds equipped with a $4$-dimensional cohomology class \cite{DiaconescuMooreWitten}.

\section{The Groups $\Oo_k^{String}(BE_8\times BE_8)$ for $k<15$}

The earlier analysis goes though {\em{mutatis mutandis}}, showing that it will suffice to compute $\widetilde{tmf}_k K(\Z\times\Z,4)$. Since we have a stable splitting
\[K(\Z\times\Z,4)=K(\Z,4)\vee K(\Z,4) \vee K(\Z,4)\wedge K(\Z,4),\] the homology computation will be two copies of the results of the previous section together with a factor coming from the smash square of $K(\Z,4)$.

The cohomology of this is easily determined. Let $I$ denote the augmentation ideal of the graded algebra $H^\ast(X;\F_p)$ for any simply connected space $X$. Since we are working over a field, the K\"unneth theorem says that
\[
H^\ast(X\times X;\F_p)=H^\ast(X;\F_p)\otimes H^\ast(X;\F_p).
\]
The reduced cohomology of $X\wedge X$ is therefore given by the ideal $I\otimes I$ sitting in this algebra, and the action of the Steenrod algebra is given by the Cartan formula and the action on $H^\ast(X;\F_p)$. In particular, since the augmentation ideal begins in degree $4$, there are very few possible classes in $I\otimes I$ below degree $16$.

\subsection{Computation with $6$ Inverted}
For degree reasons, the Atiyah-Hirzebruch spectral sequence collapses, and we quickly conclude the following theorem.

\begin{thm}\label{thm:BE8BE8odd}
For $k<15$, we have
\[
\Oo_k^{String}(BE_8\wedge BE_8)[\tfrac{1}{6}]=\begin{cases}
\Z[\tfrac{1}{6}] & \text{when }k=8, \\
\Z[\tfrac{1}{6}]^2 & \text{when }k=12, \\
0 & \text{otherwise.}
\end{cases}
\]
\end{thm}

\subsection{The Case of $p=3$}

Since we have a stable splitting of the form
\[
K(\Z,4)=A \vee B,
\]
where the bottom cell of $A$ is in dimension $4$ while that of $B$ is in $8$, we have a similar splitting
\[
K(\Z,4)\wedge K(\Z,4)=(A\wedge A)\vee (A\wedge B)\vee (B\wedge A) \vee (B\wedge B).
\]

Since we are looking only through dimension $15$, the cells in dimensions larger than $9$ in $A$ play no role. The $9$-skeleton of $A$ is a desuspension of $B$, and this means that for purposes of our computation, each summand of the right-hand side is a suspension of spectra of the form $A\wedge A$. On these, there is an action of $\Z/2$ given by interchanging the factors. This again allows us to split each factor into eigenspaces. Working through all of the computations allows us to easily compute $\tilde{H}^\ast(K(\Z,4)\wedge K(\Z,4))$, and this is presented in Figure~\ref{fig:ModuleZZat3}. We remark that here the bottom cell is in dimension $8$, meaning that the figure contains more information than is necessary.

\begin{figure}[ht]
\centering
\includegraphics{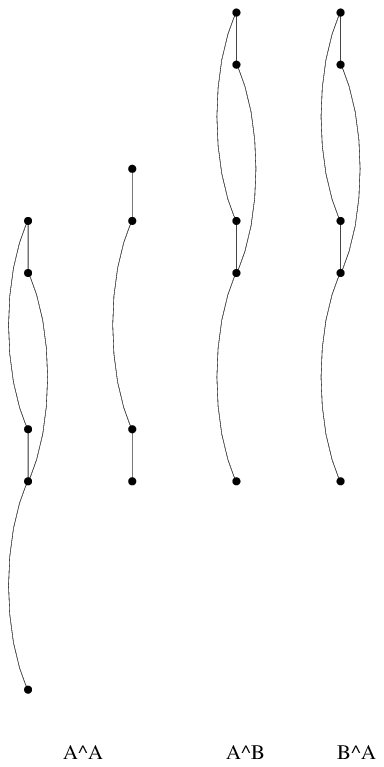}
\caption{$\tilde{H}^k\big(K(\Z,4)\wedge K(\Z,4);\F_3\big)$ for $k\leq 15$}
\label{fig:ModuleZZat3}
\end{figure}

The above analysis shows that the splitting drawn for the modules is reflected topologically. This means that there are no possible differentials or extensions linking the summands. A fairly straightforward computation yields the relevant $\Ext$ groups, and the result is given in Figure~\ref{fig:ExtZZat3}.

\begin{figure}[ht]
\centering
\includegraphics{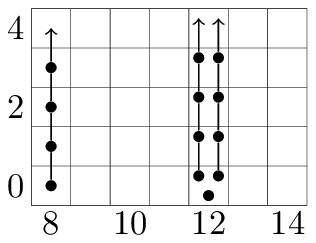}
\caption{An Adams $E_2$ term for $\widetilde{tmf}_\ast \big(K(\Z,4)\wedge K(\Z,4)\big)$ at $3$}
\label{fig:ExtZZat3}
\end{figure}

We see in particular that there are no possible differentials or extensions. This allows us to conclude the following theorem.

\begin{thm}\label{thm:BE8BE83}
Through dimension $15$,
\[
\Oo_k^{String}(BE_8\wedge BE_8)\otimes\Z_{(3)}=\begin{cases}
\Z_{(3)} & \text{when }k=8, \\
\Z_{(3)}^2\oplus \Z/3 & \text{when }k=12, \\
0 & \text{otherwise.}
\end{cases}
\]
\end{thm}

\subsection{The Case of $p=2$}

For $p=2$, the analysis for $K(\Z,4)\wedge K(\Z,4)$ is much simpler than that for $K(\Z,4)$. Computing the action of the Steenrod algebra in the cohomology is very easy using the method described above. However, given the number of cells, it is also incredibly tedious. Bruner's Ext program package contains a program for computing the tensor product of two $\A(2)$ modules, allowing us to automate this computation. Combining this computation with the $\Ext$ program allows us to compute the relevant groups. These are depicted in Figure~\ref{fig:ExtZZat2}.

\begin{figure}[ht]
\centering
\includegraphics{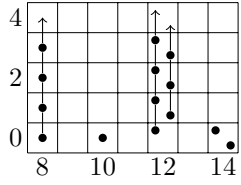}
\caption{The Adams $E_2$ term for $\widetilde{tmf}_k \big(K(\Z,4)\wedge K(\Z,4)\big)$, $k<15$}
\label{fig:ExtZZat2}
\end{figure}

For degree reasons, there are no possible differentials or extensions, and we conclude the following theorem.

\begin{thm}\label{thm:BE8BE82}
For $k<15$,
\[
\Oo_k^{String} (BE_8\wedge BE_8)\otimes\Z_{(2)}=\begin{cases}
\Z & \text{when }k=8, \\
\Z/2 & \text{when }k=10, \\
\Z^2 & \text{when }k=12, \\
(\Z/2)^2 & \text{when }k=14, \\
0 & \text{otherwise.}
\end{cases}
\]
\end{thm}

Just as before, Theorems~\ref{thm:BE8BE8odd}, \ref{thm:BE8BE83}, and \ref{thm:BE8BE82} together yield Theorem~\ref{thm:MStringBE8BE8}. Comparing the localizations yields the table from that theorem. We do remark that while the map
\[
\Oo_k^{String} BE_8\to \Oo_k^{Spin} BE_8
\]
is injective for $k<15$, the map
\[
\Oo_k^{String} (BE_8\times BE_8)\to \Oo_k^{Spin} (BE_8\times BE_8)
\]
is not.

\bibliographystyle{my-h-elsevier}

\begin{thebibliography}{10}


\bibitem{AndersonBrownPetersonSpin}
D.~W. Anderson, E.~H. Brown, Jr., and F.~P. Peterson,
The structure of the Spin cobordism ring,
Ann. of Math. (2) {\bf{86}} (1967), 271--298.

\bibitem{AndoHopkinsStricklandHinfty}
Matthew Ando, Michael~J. Hopkins, and Neil~P. Strickland,
The sigma orientation is an {$H\sb \infty$} map,
Amer. J. Math. {\bf{126}} (2004), no.~2, 247--334.

\bibitem{BahriMahowaldMString}
A.~P. Bahri and M.~E. Mahowald,
A direct summand in {$H\sp{\ast} (M{\rm O}\langle 8\rangle ,\,Z\sb{2})$},
Proc. AMS {\bf{78}} (1980), no.~2, 295--298.

\bibitem{BakerLazarev}
Andrew Baker and Andrej Lazarev,
On the Adams spectral sequence for {$R$}-modules,
Algebr. Geom. Topol. {\bf{1}} (2001), 173--199

\bibitem{BottSamelson}
Raoul Bott and Hans Samelson,
Applications of the theory of Morse to symmetric spaces,
Amer. J. Math. {\bf{80}} (1958), 964--1029.

\bibitem{BrunerExt90s}
Robert~R. Bruner,
{${\rm Ext}$} in the nineties,
Algebraic topology (Oaxtepec, 1991), Contemp. Math., vol. 146, Amer. Math. Soc., Providence, RI, 1993, pp.~71--90.

\bibitem{CartanKpin}
Henri Cartan,
Sur les groupes d'Eilenberg-MacLane. {II},
Proc. Nat. Acad. Sci. U. S. A. {\bf{40}} (1954), 704--707.

\bibitem{DiaconescuMooreWitten}
Duiliu-Emanuel Diaconescu, Gregory Moore, and Edward Witten,
{$E\sb 8$} gauge theory, and a derivation of {$K$}-theory from M-theory,
Adv. Theor. Math. Phys. {\bf{6}} (2002), no.~6, 1031--1134 (2003).

\bibitem{EdwardsMSpinBE8BE8}
Steven~R. Edwards,
On the spin bordism of {$B(E\sb 8\times E\sb 8)$},
Illinois J. Math. {\bf{35}} (1991), no.~4, 683--689.

\bibitem{JNKFBSpin}
John N.~K. Francis,
Spin bordism of {$BSpin$} and {$K({Z},4)$}: Integrality and index theory,
Undergraduate Thesis, Harvard University.

\bibitem{HilltmfBSigma3}
Michael~A. Hill,
The 3-local {${\rm tmf}$}-homology of {$B\Sigma\sb 3$},
Proc. AMS {\bf{135}} (2007), no.~12, 4075--4086 (electronic).

\bibitem{HopkinsMahowaldEO2}
Michael~J. Hopkins and Mark Mahowald,
From elliptic curves to homotopy theory,
Available on the Hopf Archive, 1998.

\bibitem{HMtmf}
Michael~J. Hopkins and Haynes~R. Miller,
Elliptic curves and stable homotopy {I},
In preparation.

\bibitem{HoveyRavenelBO8}
Mark~A. Hovey and Douglas~C. Ravenel,
The {$7$}-connected cobordism ring at {$p=3$},
Trans. AMS {\bf{347}} (1995), no.~9, 3473--3502.

\bibitem{LandweberStongBilinearForm}
Peter~S. Landweber and Robert~E. Stong,
A bilinear form for Spin manifolds,
Trans. AMS {\bf{300}} (1987), no.~2, 625--640.

\bibitem{GreenBook}
Douglas~C. Ravenel,
Complex cobordism and stable homotopy groups of spheres,
Pure and Applied Mathematics, vol. 121, Academic Press Inc., Orlando, FL, 1986.

\bibitem{Rezk512Notes}
Charles Rezk,
Supplimentary notes for math 512,
{http://www.math.uiuc.edu/~rezk/papers.html}.

\bibitem{SerreKpin}
Jean-Pierre Serre,
Sur les groupes d'Eilenberg-MacLane,
C. R. Acad. Sci. Paris {\bf{234}} (1952), 1243--1245.

\bibitem{StongMSpinKZ4}
R.~E. Stong,
Appendix: calculation of {$\Omega\sp {{\rm Spin}}\sb
  {11}(K(Z,4))$},
Workshop on unified string theories (Santa Barbara, Calif., 1985),
World Sci. Publishing, Singapore, 1986, pp.~430--437.

\bibitem{StongBOn}
Robert~E. Stong,
Determination of {$H\sp{\ast} ({\rm BO}(k,\cdots,\infty),Z\sb{2})$}\ and {$H\sp{\ast} ({\rm BU}(k,\cdots,\infty ),Z\sb{2})$},
Trans. AMS {\bf{107}} (1963), 526--544.


\end{thebibliography}

\end{document}